\documentclass[12pt]{amsart}
\usepackage[latin1]{inputenc}
\usepackage{txfonts}      
\usepackage{amssymb}
\usepackage{eucal}
\usepackage{amsmath}
\usepackage{amscd}
\usepackage{xcolor}
\usepackage{multicol}
\usepackage[all]{xy}           
\usepackage{graphicx}
\usepackage{color}
\usepackage{colordvi}
\usepackage{xspace}
\usepackage{tikz}
\usepackage{makecell}
\usepackage{appendix}
\usepackage{amsthm}

\usepackage{ifpdf}
\ifpdf
\usepackage[colorlinks,final,backref=page,hyperindex]{hyperref}
\else
\usepackage[colorlinks,final,backref=page,hyperindex,hypertex]{hyperref}
\fi

\usepackage[active]{srcltx} 

\begin{document}


\newtheorem{thm}{Theorem}[section]
\newtheorem{lem}[thm]{Lemma}
\newtheorem{cor}[thm]{Corollary}
\newtheorem{pro}[thm]{Proposition}
\theoremstyle{definition}
\newtheorem{defi}[thm]{Definition}
\newtheorem{ex}[thm]{Example}
\newtheorem{rmk}[thm]{Remark}
\newtheorem{pdef}[thm]{Proposition-Definition}
\newtheorem{condition}[thm]{Condition}

\renewcommand{\labelenumi}{{\rm(\alph{enumi})}}
\renewcommand{\theenumi}{\alph{enumi}}

\newcommand {\emptycomment}[1]{} 

\newcommand{\nc}{\newcommand}
\newcommand{\delete}[1]{}

\nc{\tred}[1]{\textcolor{red}{#1}}
\nc{\tblue}[1]{\textcolor{blue}{#1}}
\nc{\tgreen}[1]{\textcolor{green}{#1}}
\nc{\tpurple}[1]{\textcolor{purple}{#1}}
\nc{\tgray}[1]{\textcolor{gray}{#1}}
\nc{\torg}[1]{\textcolor{orange}{#1}}
\nc{\tmag}[1]{\textcolor{magenta}}
\nc{\btred}[1]{\textcolor{red}{\bf #1}}
\nc{\btblue}[1]{\textcolor{blue}{\bf #1}}
\nc{\btgreen}[1]{\textcolor{green}{\bf #1}}
\nc{\btpurple}[1]{\textcolor{purple}{\bf #1}}

\nc{\revise}[1]{\textcolor{blue}{#1}}


\nc{\tforall}{\ \ \text{for all }}
\nc{\hatot}{\,\widehat{\otimes} \,}
\nc{\complete}{completed\xspace}
\nc{\wdhat}[1]{\widehat{#1}}

\nc{\ts}{\mathfrak{p}}
\nc{\mts}{c_{(i)}\ot d_{(j)}}

\nc{\NA}{{\bf NA}}
\nc{\LA}{{\bf Lie}}
\nc{\CLA}{{\bf CLA}}

\nc{\cybe}{CYBE\xspace}
\nc{\nybe}{NYBE\xspace}
\nc{\ccybe}{CCYBE\xspace}

\nc{\preliecom}{pre-Lie commutative\xspace}
\nc{\transpreliecom}{transposed pre-Lie commutative\xspace}
\nc{\transNovikovpoisson}{transposed Novikov-Poisson\xspace}
\nc{\transdiffnovikovpoisson}{transposed differential Novikov-poisson\xspace}
\nc{\diffnovikovpoisson}{differential Novikov-Poisson\xspace}
\nc{\preliepoisson}{pre-Lie Poisson\xspace}
\nc{\transpreliepoisson}{transposed pre-Lie poisson\xspace}

\nc{\calb}{\mathcal{B}}
\nc{\rk}{\mathrm{r}}
\newcommand{\g}{\mathfrak g}
\newcommand{\h}{\mathfrak h}
\newcommand{\pf}{\noindent{$Proof$.}\ }
\newcommand{\frkg}{\mathfrak g}
\newcommand{\frkh}{\mathfrak h}
\newcommand{\Id}{\rm{Id}}
\newcommand{\gl}{\mathfrak {gl}}
\newcommand{\ad}{\mathrm{ad}}
\newcommand{\add}{\frka\frkd}
\newcommand{\frka}{\mathfrak a}
\newcommand{\frkb}{\mathfrak b}
\newcommand{\frkc}{\mathfrak c}
\newcommand{\frkd}{\mathfrak d}
\newcommand {\comment}[1]{{\marginpar{*}\scriptsize\textbf{Comments:} #1}}

\nc{\vspa}{\vspace{-.1cm}}
\nc{\vspb}{\vspace{-.2cm}}
\nc{\vspc}{\vspace{-.3cm}}
\nc{\vspd}{\vspace{-.4cm}}
\nc{\vspe}{\vspace{-.5cm}}

\nc{\yy}[1]{\textcolor{blue}{Yanyong: #1}}
\nc{\disp}[1]{\displaystyle{#1}}
\nc{\bin}[2]{ (_{\stackrel{\scs{#1}}{\scs{#2}}})}  
\nc{\binc}[2]{ \left (\!\! \begin{array}{c} \scs{#1}\\
    \scs{#2} \end{array}\!\! \right )}  
\nc{\bincc}[2]{  \left ( {\scs{#1} \atop
    \vspace{-.5cm}\scs{#2}} \right )}  
\nc{\ot}{\otimes}
\nc{\sot}{{\scriptstyle{\ot}}}
\nc{\otm}{\overline{\ot}}
\nc{\ola}[1]{\stackrel{#1}{\la}}

\nc{\scs}[1]{\scriptstyle{#1}} \nc{\mrm}[1]{{\rm #1}}

\nc{\dirlim}{\displaystyle{\lim_{\longrightarrow}}\,}
\nc{\invlim}{\displaystyle{\lim_{\longleftarrow}}\,}

\nc{\bfk}{{\bf k}} \nc{\bfone}{{\bf 1}}
\nc{\rpr}{\circ}
\nc{\dpr}{{\tiny\diamond}}
\nc{\rprpm}{{\rpr}}

\nc{\mmbox}[1]{\mbox{\ #1\ }} \nc{\ann}{\mrm{ann}}
\nc{\Aut}{\mrm{Aut}} \nc{\can}{\mrm{can}}
\nc{\twoalg}{{two-sided algebra}\xspace}
\nc{\colim}{\mrm{colim}}
\nc{\Cont}{\mrm{Cont}} \nc{\rchar}{\mrm{char}}
\nc{\cok}{\mrm{coker}} \nc{\dtf}{{R-{\rm tf}}} \nc{\dtor}{{R-{\rm
tor}}}
\renewcommand{\det}{\mrm{det}}
\nc{\depth}{{\mrm d}}
\nc{\End}{\mrm{End}} \nc{\Ext}{\mrm{Ext}}
\nc{\Fil}{\mrm{Fil}} \nc{\Frob}{\mrm{Frob}} \nc{\Gal}{\mrm{Gal}}
\nc{\GL}{\mrm{GL}} \nc{\Hom}{\mrm{Hom}} \nc{\hsr}{\mrm{H}}
\nc{\hpol}{\mrm{HP}}  \nc{\id}{\mrm{id}} \nc{\im}{\mrm{im}}

\nc{\incl}{\mrm{incl}} \nc{\length}{\mrm{length}}
\nc{\LR}{\mrm{LR}} \nc{\mchar}{\rm char} \nc{\NC}{\mrm{NC}}
\nc{\mpart}{\mrm{part}} \nc{\pl}{\mrm{PL}}
\nc{\ql}{{\QQ_\ell}} \nc{\qp}{{\QQ_p}}
\nc{\rank}{\mrm{rank}} \nc{\rba}{\rm{RBA }} \nc{\rbas}{\rm{RBAs }}
\nc{\rbpl}{\mrm{RBPL}}
\nc{\rbw}{\rm{RBW }} \nc{\rbws}{\rm{RBWs }} \nc{\rcot}{\mrm{cot}}
\nc{\rest}{\rm{controlled}\xspace}
\nc{\rdef}{\mrm{def}} \nc{\rdiv}{{\rm div}} \nc{\rtf}{{\rm tf}}
\nc{\rtor}{{\rm tor}} \nc{\res}{\mrm{res}} \nc{\SL}{\mrm{SL}}
\nc{\Spec}{\mrm{Spec}} \nc{\tor}{\mrm{tor}} \nc{\Tr}{\mrm{Tr}}
\nc{\mtr}{\mrm{sk}}

\nc{\ab}{\mathbf{Ab}} \nc{\Alg}{\mathbf{Alg}}

\nc{\BA}{{\mathbb A}} \nc{\CC}{{\mathbb C}} \nc{\DD}{{\mathbb D}}
\nc{\EE}{{\mathbb E}} \nc{\FF}{{\mathbb F}} \nc{\GG}{{\mathbb G}}
\nc{\HH}{{\mathbb H}} \nc{\LL}{{\mathbb L}} \nc{\NN}{{\mathbb N}}
\nc{\QQ}{{\mathbb Q}} \nc{\RR}{{\mathbb R}} \nc{\BS}{{\mathbb{S}}} \nc{\TT}{{\mathbb T}}
\nc{\VV}{{\mathbb V}} \nc{\ZZ}{{\mathbb Z}}


\nc{\calao}{{\mathcal A}} \nc{\cala}{{\mathcal A}}
\nc{\calc}{{\mathcal C}} \nc{\cald}{{\mathcal D}}
\nc{\cale}{{\mathcal E}} \nc{\calf}{{\mathcal F}}
\nc{\calfr}{{{\mathcal F}^{\,r}}} \nc{\calfo}{{\mathcal F}^0}
\nc{\calfro}{{\mathcal F}^{\,r,0}} \nc{\oF}{\overline{F}}
\nc{\calg}{{\mathcal G}} \nc{\calh}{{\mathcal H}}
\nc{\cali}{{\mathcal I}} \nc{\calj}{{\mathcal J}}
\nc{\call}{{\mathcal L}} \nc{\calm}{{\mathcal M}}
\nc{\caln}{{\mathcal N}} \nc{\calo}{{\mathcal O}}
\nc{\calp}{{\mathcal P}} \nc{\calq}{{\mathcal Q}} \nc{\calr}{{\mathcal R}}
\nc{\calt}{{\mathcal T}} \nc{\caltr}{{\mathcal T}^{\,r}}
\nc{\calu}{{\mathcal U}} \nc{\calv}{{\mathcal V}}
\nc{\calw}{{\mathcal W}} \nc{\calx}{{\mathcal X}}
\nc{\CA}{\mathcal{A}}

\nc{\fraka}{{\mathfrak a}} \nc{\frakB}{{\mathfrak B}}
\nc{\frakb}{{\mathfrak b}} \nc{\frakd}{{\mathfrak d}}
\nc{\oD}{\overline{D}}
\nc{\frakF}{{\mathfrak F}} \nc{\frakg}{{\mathfrak g}}
\nc{\frakm}{{\mathfrak m}} \nc{\frakM}{{\mathfrak M}}
\nc{\frakMo}{{\mathfrak M}^0} \nc{\frakp}{{\mathfrak p}}
\nc{\frakS}{{\mathfrak S}} \nc{\frakSo}{{\mathfrak S}^0}
\nc{\fraks}{{\mathfrak s}} \nc{\os}{\overline{\fraks}}
\nc{\frakT}{{\mathfrak T}}
\nc{\oT}{\overline{T}}
\nc{\frakX}{{\mathfrak X}} \nc{\frakXo}{{\mathfrak X}^0}
\nc{\frakx}{{\mathbf x}}
\nc{\frakTx}{\frakT}      
\nc{\frakTa}{\frakT^a}        
\nc{\frakTxo}{\frakTx^0}   
\nc{\caltao}{\calt^{a,0}}   
\nc{\ox}{\overline{\frakx}} \nc{\fraky}{{\mathfrak y}}
\nc{\frakz}{{\mathfrak z}} \nc{\oX}{\overline{X}}


\title[Nilpotency and Frattini theory for transposed Poisson algebras]{Nilpotency and Frattini theory for transposed Poisson algebras}

\author{Jiarou Jin}
\address{School of Mathematics, Hangzhou Normal University,
Hangzhou, 311121, China}
\email{jrjin@stu.hznu.edu.cn}

\author{Yanyong Hong (corresponding author)}
\address{School of Mathematics, Hangzhou Normal University,
Hangzhou, 311121, China}
\email{yyhong@hznu.edu.cn}



\begin{abstract}
We develop the theory of nilpotency and the Frattini theory for transposed Poisson algebras. The lower central series is shown to admit a simplified form, and an analogue of Engel's theorem is established: a finite-dimensional transposed Poisson algebra is nilpotent precisely when the left multiplication operators in both the associative and the Lie structures are nilpotent. Constructions of nilpotent and solvable algebras via tensor products and derivations are given. For a finite-dimensional Lie-nilpotent transposed Poisson algebra, we prove that the derived Lie subalgebra is a nilpotent ideal, which implies that the nilpotent radical coincides with the associative radical. In the framework of Frattini theory, we show that the Frattini subalgebra is always contained in the derived algebra and the Frattini ideal is associative nilpotent. When the algebra is nilpotent, all maximal subalgebras are ideals and the Frattini subalgebra equals the derived algebra. Conversely, for a Lie-nilpotent transposed Poisson algebra, if all maximal subalgebras are ideals, the algebra either is nilpotent or decomposes as a direct sum of a one-dimensional algebra generated by an idempotent and the nilpotent radical; if the Frattini subalgebra equals the derived algebra, the algebra is necessarily nilpotent. We also prove that the zero socle coincides with the nilpotent radical, and when the Frattini ideal is zero, the algebra splits into a subalgebra and its zero socle; in the Lie-nilpotent case this subalgebra is abelian as a Lie algebra.
\end{abstract}

\maketitle

\vspace{-1.2cm}

\tableofcontents

\vspace{-1.2cm}

\allowdisplaybreaks

\section{Introduction}
\vspace{-.2cm}
Poisson algebras, which originated in the 1970s from the study of Poisson geometry~\cite{poisson1,poisson2}, have become fundamental algebraic structures with deep applications in symplectic geometry, quantization theory, integrable systems, and quantum groups. A Poisson algebra is equipped with two compatible operations: a commutative associative multiplication and a Lie bracket, related by a Leibniz rule. Recently, a dual notion, termed a transposed Poisson algebra, was introduced in \cite{Bai} by exchanging the roles of the two operations in this Leibniz rule. Explicitly, a transposed Poisson algebra is a triple $(P, \cdot, [\cdot,\cdot])$ where $(P, \cdot)$ is a commutative associative algebra, $(P, [\cdot,\cdot])$ is a Lie algebra, and the compatibility condition
\begin{eqnarray}
&&\label{transposedpoissonlu}
2z \cdot [x, y] = [z \cdot x, y] + [x, z \cdot y] \quad \text{for all}\: x, y, z \in P.
\end{eqnarray}
holds. For a transposed Poisson algebra $(P, \cdot, [\cdot,\cdot])$, if $(P, \cdot)$ or $(P, [\cdot,\cdot])$ is trivial, then we say that $(P, \cdot, [\cdot,\cdot])$ is {\bf trivial}. Transposed Novikov-Poisson algebras share many key properties with Poisson algebras, such as closure under tensor products and Koszul self-duality as an operad, and are closely connected to other algebraic structures including Novikov-Poisson algebras, $3$-Lie algebras, Gel¡¯fand-Dorfman algebras and algebras of Jordan brackets \cite{Bai, S, simple-simple}. A significant link between $\frac{1}{2}$-derivations of Lie algebras and transposed Poisson algebras, established in \cite{1-2}, has led to the classification of transposed Poisson algebra structures on many important Lie algebras, such as simple, Witt, Virasoro, and various other infinite-dimensional Lie algebras \cite{YH, KK, KKS, ZSZ, KK2}. Simple transposed Poisson algebras have been studied in \cite{simple-simple} and a bialgebra theory of such algebras have been presented in \cite{LB}.

A key feature of finite-dimensional Poisson algebras is the absence of non-trivial simple structures~\cite{non-trivial simple structures}, which naturally directs attention to the study of nilpotency and solvability. Moreover, the work~\cite{Poisson algebras-non-associative algebras} presented Poisson algebras in terms of non-associative algebras via polarization and depolarization, making it natural to involve both the associative and Lie products when defining nilpotency and solvability. Accordingly, the study of nilpotent and solvable Poisson algebras has progressed considerably in recent years. In~\cite{On the solvable Poisson algebras}, Fern\'andez Ouaridi and Omirov established Engel's and Lie's theorems for Poisson algebras, and further investigated the construction of nilpotent and solvable Poisson algebras via tensor products and generalized Jacobians, as well as study Poisson algebras on maximal solvable extensions of nilpotent Lie algebras. In addition, solvability and constructions of generalized Poisson algebras were investigated in~\cite{generalized Poisson}. These developments make it very natural to investigate solvability and nilpotency for transposed Poisson algebras.

The investigation of nilpotency and solvability in algebraic structures is deeply intertwined with Frattini theory. Originating in group theory, the Frattini subgroup captures the notion of non-generators and provides a powerful nilpotency criterion: a group is nilpotent if and only if its Frattini subgroup contains its derived subgroup \cite{ Frattini group}. This theory has been successfully extended to numerous algebraic settings, including Lie algebras~\cite{Frattini of Lie algebras, Lie2, Lie3}, $n$-Lie algebras~\cite{Frattini n-Lie}, associative algebras~\cite{Frattini asso}, general non-associative algebras~\cite{Frattini non-asso, Frattini of algebras}, Leibniz algebras, bicommutative algebras, assosymmetric algebras, Novikov algebras~\cite{Frattini Novikov}, restricted Lie algebras~\cite{Frattini restricted}, restricted Lie Superalgebras~\cite{Frattini restrictedsuper}, evolution algebras~\cite{Frattini evolution}, and most recently by Towers extended to Poisson algebras~\cite{Frattini of dialgebras}. Therefore, it is natural to consider to develop a Frattini theory for transposed Poisson algebras.\par

In this paper, we address these questions and initiate the systematic study of the nilpotency structure and Frattini theory of transposed Poisson algebras. Our main contributions are as follows. We first show that the lower central series of a transposed Poisson algebra admits a particularly convenient form, and we identify several natural ideals arising from the compatibility condition. We establish an analogue of Engel¡¯s theorem, proving that a finite-dimensional transposed Poisson algebra is nilpotent if and only if the left multiplication operators in both the associative and Lie structures are nilpotent. We further provide constructions of nilpotent and solvable transposed Poisson algebras via tensor products and from commutative associative algebras with derivations. A key structural result is that for a finite-dimensional Lie-nilpotent transposed Poisson algebra, its derived Lie subalgebra $[P,P]$ is a nilpotent ideal, which forces its nilpotent radical to coincide with the associative nilpotent radical.

Turning to Frattini theory, we prove that the Frattini subalgebra is always contained in the derived algebra, and the Frattini ideal is always associative nilpotent. We also give some bounds for them(see Proposition~\ref{F(P)lower bounds} and Corollary~\ref{F(P)lower bounds-Nil}). For nilpotent algebras, we obtain a precise analogue of the group-theoretic and Lie-algebraic results: all maximal subalgebras are ideals, and the Frattini subalgebra equals the derived algebra. We then explore the converse direction for Lie-nilpotent algebras, showing that if all maximal subalgebras are ideals, the algebra decomposes as a direct sum of a one-dimensional algebra generated by an idempotent and the nilpotent radical; if the Frattini subalgebra equals the derived algebra, the algebra must be nilpotent. Finally, we relate these concepts to the socle, proving that the zero socle coincides with the nilpotent radical, and we analyze the structure of transposed Poisson algebras with a vanishing Frattini ideal, showing they decompose into a subalgebra and its zero socle, with the subalgebra being abelian as a Lie algebra in the Lie-nilpotent case.

The paper is organized as follows. In Section 2, we first introduce the definitions and fundamental properties of nilpotent and solvable transposed Poisson algebras, along with several associated ideals and subalgebras. Moreover, we explore several properties of idempotents in transposed Poisson algebras.
Section 3 is devoted to an analogue of Engel's theorem for transposed Poisson algebras, along with some constructions of nilpotent and solvable algebras. We then prove that in a finite-dimensional Lie-nilpotent transposed Poisson algebra, the derived Lie subalgebra is always a nilpotent ideal, which allows us to describe the nilpotent radical in detail. In Section 4, we develop the Frattini theory for transposed Poisson algebras. We introduce the Frattini subalgebra and Frattini ideal, establish their fundamental properties, and study the interplay among nilpotency, the condition that all maximal subalgebras are ideals, and the equality between the Frattini subalgebra and the derived subalgebra. Finally, we analyze complements of the zero socle when the Frattini ideal vanishes.

\noindent{\bf Notations.} Throughout this paper, let  $\bf k$ be a field whose characteristic is not equal to $2$. All tensors over ${\bf k}$ are denoted by $\otimes$.
Let $\dot{+}$ denote a direct sum of the underlying vector spaces and $\oplus$ denote a direct sum of the underlying algebras.
Denote by $\mathbb{C}$, $\mathbb{Z}$, $\mathbb{N}$ and $\mathbb{Z}_+$ the sets of complex numbers, integer numbers, non-negative integers and positive integers respectively.

\section{Some related properties of transposed Poisson algebras}
This section is devoted to several distinguished ideals, as well as the fundamental properties of nilpotency and idempotent elements in transposed Poisson algebras.
\subsection{Basic notions and properties}
Recall that a {\bf dialgebra} $(P,\cdot, [\cdot,\cdot])$ is a vector space $P$ with two binary operations $\cdot$ and $[\cdot,\cdot]$. Note that associative dialgebras were originally defined by Loday in the 1990s~\cite{assodialgebradefi}. A \textbf{subalgebra} of a dialgebra $(P,\cdot, [\cdot,\cdot])$ is a linear subspace \(A\) satisfying
$A \cdot A + [A,A] \subseteq A$. A subalgebra \(I\) of $(P,\cdot, [\cdot,\cdot])$ is an \textbf{ideal} if
$I \cdot P + P \cdot I + [I,P] + [P,I] \subseteq I$.
A \textbf{abelian subalgebra} (resp. \textbf{abelian ideal}) of $(P,\cdot, [\cdot,\cdot])$ is a subalgebra (resp. ideal) \(Z\) such that $Z \cdot Z + [Z,Z] = 0$.

\begin{defi}~\cite{dialgebra nilsol}
Let $(P,\cdot,[\cdot,\cdot])$ be a dialgebra and $A$ be a subalgebra of $P$. The \textbf{derived series} of $A$ is defined by
\begin{eqnarray*}
&&A^{(0)} := A, \qquad A^{(n+1)} := A^{(n)}\!\cdot\! A^{(n)} + [A^{(n)}, A^{(n)}]\; \quad\text{for all}\:n\geq 0.
\end{eqnarray*}
The dialgebra $(A,\cdot,[\cdot,\cdot])$ is called \textbf{solvable} if there exists $m \ge 0$ such that $A^{(m)} = 0$.\par
The \textbf{lower central series} of $A$ is the sequence of $A$ given by
\begin{eqnarray*}
A^{0} := A, \qquad
A^{n+1} := \sum_{i=0}^{n} \bigl( A^{i}\!\cdot\! A^{n-i} + [A^{i}, A^{n-i}] \bigr) \quad\text{for all}\:n\geq 0.
\end{eqnarray*}
The dialgebra $(A,\cdot,[\cdot,\cdot])$ is called \textbf{nilpotent} if there exists $m \ge 0$ such that $A^{m} = 0$.
\end{defi}

\begin{defi}~\cite{Frattini of dialgebras}
Let $(P,\cdot,[\cdot,\cdot])$ be an $n$-dimensional dialgebra. The dialgebra $P$ is called \textbf{supersolvable} if there is a flag
\begin{eqnarray*}
&&0 = P_0 \subset P_1 \subset \ldots \subset P_n = P,
\end{eqnarray*}
where $P_i$ is an $i$-dimensional ideal of $P$ for $1 \leq i \leq n$.
\end{defi}

Note that a transposed Poisson algebra is a dialgebra.
Given a transposed Poisson algebra $(P,\cdot,[\cdot,\cdot])$. We use the notation $P_A$ and $P_L$ to denote the associative algebra structure $(P,\cdot)$ and the Lie algebra structure $(P,[\cdot,\cdot])$ respectively. Furthermore, for each $x\in P$, define $P_x$ and $Q_x\in\operatorname{End}(A)$ by $P_x(y)=x\cdot y$ and $Q_x(y)=[x,y]$ respectively for all $y\in P$.

It has been proved that the nilpotency implies the solvability for a dialgebra~\cite{On the solvable Poisson algebras}. The following proposition shows that nilpotency implies supersolvability for a finite-dimensional transposed Poisson algebra over an algebraically closed field of characteristic zero.

\begin{pro}\label{nil-supersol}
Let $(P,\cdot,[\cdot,\cdot])$ be a finite-dimensional nilpotent transposed Poisson algebra over an algebraically closed field of characteristic zero. Then $P$ is supersolvable.
\end{pro}

\begin{proof}
We first claim that there exists a non-zero common eigenvector $v$ of $P_x$ and $Q_x$ for all $x\in P$. Since $P$ is nilpotent, $P_L$ is solvable. Then we can choose a chain of Lie ideals as follows:
\begin{eqnarray*}
0=V_0\subset V_1\subset \cdots\subset V_n=P_L,
\end{eqnarray*}
 where $\dim(V_i)=i$ and $V_i$ is an ideal of $P_L$ for each $i$. Consider the lower central series of $P$. There exists a nonnegative integer $m$ such that $P^{m}\neq0$ but $P^{m+1}=0$. Choose an index $j$ for which $\dim(P^{m}\cap V_j)=1$. For any non-zero $v\in P^{m}\cap V_j$, we have $Q_x(v)$, $P_x(v)\in P\cdot P^{m}+[P, P^{m}]\subseteq P^{m+1}=0$ for all $x\in P$. Thus $v$ is a common eigenvector of $P_x$ and $Q_x$ for all $x\in P$, proving the claim. Hence, the supersolvability of $P$ is clear by induction on $\dim(P)$. This completes the proof.
\delete{Now we prove supersolvability by induction on $\dim(P)$. If $\dim(P)=1$ the statement is trivial. Assume $\dim(P)>1$ and let $v$ be a common eigenvector as above. Set $I={\bf k}v$, then $I$ is a $1$-dimensional ideal of $P$. By the induction hypothesis, the quotient $P/I$ is supersolvable, so we can take a supersolvable flag
\begin{eqnarray*}
&&0=\bar{U}_0\subseteq\bar{U}_1\subseteq\cdots\subseteq\bar{U}_{n-1}=P/I
\end{eqnarray*}
with $\dim(\bar{U}_i)=i$. Then
\begin{eqnarray*}
&&0\subseteq I\subseteq\pi^{-1}(\bar{U}_1)\subseteq\cdots\subseteq\pi^{-1}(\bar{U}_{n-1})=P.
\end{eqnarray*}
is a flag of ideals of $P$ with dimensions $0,1,\dots,n$, so $P$ is supersolvable.}
\end{proof}
\delete{\begin{rmk}
However, the converse of Proposition~\ref{nil-supersol} is not necessarily true.
\end{rmk}

\begin{ex}\label{ex2}
Let $P$ be a $2$-dimensional transposed Poisson algebra over $\mathbb{C}$ with basis $\{e_1, e_2\}$, where the non-zero commutative associative product and Lie bracket are given by
\begin{eqnarray*}
 && e_1 \cdot e_1 = e_2,\quad [e_1, e_2] = e_2
\end{eqnarray*}
respectively. According to the classification in \cite{Bai}, this defines a transposed Poisson algebra.\par
A direct computation shows that there is a flag $0=P_0\subseteq P_1=\mathbb{C}e_2\subseteq P_2=P$, where $P_1$ is an ideal, but $P$ itself is not nilpotent.
\end{ex}}

The lower central series of transposed Poisson algebras can be given in a more concise form.
\begin{pro}\label{nildefi1}
Let $(P,\cdot,[\cdot,\cdot])$ be a transposed Poisson algebra and $A$ be a subalgebra of $P$. Then we have $A^{k+1}=A\cdot A^{k}+[A,A^{k}]$ for each $k\geq 0$.
\end{pro}

\begin{proof}
Since $A \cdot A^{k}+[A,A^{k}]\subseteq A^{k+1}$ is obvious, we only need to prove that
\begin{eqnarray*}
A^{k+1}\subseteq A^{k}\cdot A+[A^{k},A]\qquad\text{for all }k\geq 0.
\end{eqnarray*}
We proceed by induction on $k$.
If $k=0$, the statement is obvious.
Assume that $A^{n+1}\subseteq A^{n}\cdot A+[A^{n},A]$ holds for all $0\leq n\leq k-1$.
Then we consider $A^{k+1}$. It suffices to show that
\begin{eqnarray*}
  &&A^{i}\cdot A^{k-i}+[A^{i},A^{k-i}]\subseteq A^{k}\cdot A+[A^{k},A]\qquad\text{for all }0\leq i\leq k.
\end{eqnarray*}

We prove this by a second induction on $i$.
For $i=0$ the inclusion is clear.
Suppose that $A^{j}\cdot A^{k-j}+[A^{j},A^{k-j}]\subseteq A^{k}\cdot A+[A^{k},A]$ for all $0\leq j\leq i$, and consider $A^{i+1}\cdot A^{k-i-1}$ and $[A^{i+1},A^{k-i-1}]$.
Since $i+1\leq k$, we may apply the outer inductive hypothesis to obtain
\begin{eqnarray*}
A^{i+1}\cdot A^{k-i-1}&=&\big([A,A^{i}]+A\cdot A^{i}\big) \cdot A^{k-i-1}\\
&\subseteq& [A,A^{i}]\cdot A^{k-i-1}+A\cdot (A^{i}\cdot A^{k-i-1})\\
&\subseteq& [A\cdot A^{k-i-1},A^i]+[A,A^i\cdot A^{k-i-1}]+A\cdot A^k\\
&\subseteq& [A^{k-i},A^i]+[A,A^k]+A\cdot A^k\\
&\subseteq& [A,A^k]+A\cdot A^k, 
\end{eqnarray*}
and similarly
\begin{eqnarray*}
[A^{i+1},A^{k-i-1}]&=&\big[[A,A^{i}]+A\cdot A^{i},A^{k-i-1}\big]\\
&\subseteq&\big[[A,A^{k-i-1}],A^i\big]+\big[[A^i,A^{k-i-1}],A]\big]+A\cdot [A^i,A^{k-i-1}]+[A^i,A\cdot A^{k-i-1}] \\
&\subseteq& [A^{k-i},A^i]+[A^k,A]+A\cdot A^{k}\\
&\subseteq& [A,A^k]+A\cdot A^k. 
\end{eqnarray*}
This completes the inner induction, and consequently the outer induction. This completes the proof.
\end{proof}

\delete{In the following propositions, we will identify some ideals of a transposed Possion algebra.}

\begin{pro}\label{idealpro1}
Let  \( (P, \cdot, [\cdot,\cdot]) \) be a transposed Poisson algebra and $I$ be an ideal of $P$. Then $I^k$ and $I^{(k)}$ are ideals of $P$ for each $k\geq 0$.
\end{pro}

\begin{proof}
We proceed by induction on $k$. It is clear when $k=0$. Assume that for every $t$ with $0 \le t \le k$, the subspaces $I^{t}$ and $I^{(t)}$ are ideals of $P$. Now we consider $I^{k+1}$ and $I^{(k+1)}$. Using the induction hypothesis and Proposition~\ref{nildefi1}, we obtain
\begin{eqnarray*}
[I^{k+1},P] + I^{k+1}\cdot P &=& \big[[I^{k},I] + I^{k}\cdot I,\; P\big] + \big([I^{k},I] + I^{k}\cdot I\big)\cdot P \\
&\subseteq& \big[[I^{k},I],P\big] + I\cdot[I^{k},P] + [I^{k},I\cdot P] + [I^{k}\cdot P,I] + I^{k}\cdot (I\cdot P) \\ 
&\subseteq& [I^{k},I] + I\cdot I^{k} \\
&=& I^{k+1},
\end{eqnarray*}
and similarly
\begin{eqnarray*}
[I^{(k+1)},P] + I^{(k+1)}\cdot P &=& \big[[I^{(k)},I^{(k)}] + I^{(k)}\cdot I^{(k)},\; P\big] + \big([I^{(k)},I^{(k)}] + I^{(k)}\cdot I^{(k)}\big)\cdot  P \\
&\subseteq& [I^{(k)},I^{(k)}] + I^{(k)}\cdot I^{(k)} \\
&=& I^{(k+1)}.
\end{eqnarray*}
Therefore, $I^k$ and $I^{(k)}$ are ideals of $P$ for each $k\geq 0$.
\end{proof}

\begin{pro}\label{idealpro2}
Let $(P,\cdot,[\cdot,\cdot])$ be a transposed Poisson algebra and $B,C$ be two ideals of $P$. Then $[B,C]$ is also an ideal of $P$.
\end{pro}

\begin{proof}
For any $x\in B$ and $y\in C$, by Eq.~\eqref{transposedpoissonlu}, we obtain
\begin{eqnarray*}
\big[[x,y],P\big]+[x,y] \cdot P&\subseteq& \big[[x,P],y\big]+\big[[y,P],x\big]+[x\cdot P,y]+[x,y\cdot  P]\\
&\subseteq&[B,C].
\end{eqnarray*}
This completes the proof.
\end{proof}

\begin{rmk}
If $B,C$ are two ideals of a transposed Poisson algebra $P$, then the subspace $B\cdot C$ is not necessarily an ideal of $P$. For example, by \cite{3-dimensional transposed}, $(P=\mathbb{C}e_1\oplus \mathbb{C}e_2\oplus\mathbb{C}e_3,\cdot,[\cdot,\cdot])$ is a transposed Poisson algebra with the non-zero commutative associative products and the non-zero Lie brackets given by
\begin{eqnarray*}
 && e_3\cdot e_3=e_1,\quad [e_1, e_3] = e_1+e_2,\quad [e_2,e_3]=e_2.
\end{eqnarray*}
However, it is easy to see that $P\cdot P=\mathbb{C}e_1$ is not an ideal of $P$.
\end{rmk}

\delete{
\begin{ex}
Let $(P,\cdot,[\cdot,\cdot])$ be a $3$-dimensional vector space over $\mathbb{C}$ with basis $\{e_1, e_2, e_3\}$, where the non-zero commutative associative product and Lie bracket are given by
\begin{eqnarray*}
 && e_3\cdot e_3=e_1,\quad [e_1, e_3] = e_1+e_2,\quad [e_2,e_3]=e_2
\end{eqnarray*}
respectively. According to the classification in \cite{3-dimensional transposed}, this defines a transposed Poisson algebra. However, a direct computation shows that $P\cdot P=\mathbb{C}e_1$ is not an ideal.
\end{ex}}

Next, we recall the definitions of annihilators and normalizers of subalgebras of a dialgebra.

\begin{defi}
Let $B$ be a subalgebra of a dialgebra   \( (P, \cdot, [\cdot,\cdot]) \).

\cite{Frattini of dialgebras}~The \textbf{annihilator} of $B$ in $P$, denoted $\operatorname{Ann}_{P}(B)$, is defined by $\operatorname{Ann}_{P}(B) = \{\, x \in P \mid x \cdot B + B\cdot x + [x, B] + [B, x] = 0\}$.

\cite{Abelian subalgebra of Poisson}~The \textbf{normalizer} of $B$ in $P$ is the set $ N(B) = \{x \in P \mid B\cdot x + x\cdot B+ [B, x] + [x,B]\subseteq B\}$.
\end{defi}

\delete{
\begin{defi}~\cite{Abelian subalgebra of Poisson}
If $B$ is a subalgebra of a dialgebra   \( (P, \cdot, [\cdot,\cdot]) \) , the \textbf{normalizer} of $B$ is the set $ N(B) = \{x \in P \mid P\cdot x + x\cdot P+ [P, x] + [x,P]\}$.
\end{defi}}
\begin{rmk}\label{Ann_L}
For a dialgebra \((P, \cdot, [\cdot,\cdot])\), if $B$ is an ideal, it is straightforward to verify that $\operatorname{Ann}_{P}(B)$ is also an ideal of $P$ \cite{Frattini of dialgebras}. If \((P, \cdot, [\cdot,\cdot])\) is a transposed Poisson algebra, we denote by $\operatorname{Ann}_{L}(B)$ (resp.\ $\operatorname{Ann}_{A}(B)$) the annihilator of a subalgebra $B$ in the $P_L$ (resp.\ in the $P_A$). Similar to Engel subalgebras of Lie algebras, for each $a\in P$, we define $E_P^A(a)=\{\, x \in P \mid P_a^k(x)= 0, k\gg0\}$ and $E_P^L(a)=\{\, x \in P \mid Q_a^k(x)= 0, k\gg0\}$.
\end{rmk}
\vspb
\begin{pro}\label{idealpro3}
Let $(P,\cdot, [\cdot,\cdot])$ be a transposed Poisson algebra and $B$ be subalgebra of $P$. Then $\operatorname{Ann}_L(B)$ and $E_P^A(a)$ are ideals of $P$ for all $a\in P$.
\end{pro}

\begin{proof}
It is straightforward.
\end{proof}

\begin{rmk}
 However, for each $a\in P$, $E_P^L(a)$ is not necessarily a subalgebra of $P$. For example, by  \cite{Bai}, $(P=\mathbb{C}e_1\oplus\mathbb{C}e_2,\cdot, [\cdot,\cdot])$ is a $2$-dimensional transposed Poisson algebra with the non-zero commutative associative product and the non-zero Lie bracket given by
\begin{eqnarray*}
 && e_1 \cdot e_1 = e_2,\quad [e_1, e_2] = e_2.
\end{eqnarray*}
Then we have $E_P^L(e_1)=\mathbb{C}e_1$, which is not a subalgebra of $P$.
\end{rmk}

\subsection{Idempotents in transposed Poisson algebras}
Idempotents in transposed Poisson algebras have many properties. They are related to multiplicative Hom-Lie algebra structures and play an important role in studying the Frattini theory. \vspa
\begin{pro}\label{idempotentlpro}
Let $(P,\cdot,[\cdot,\cdot])$ be a transposed Poisson algebra with an idempotent $e$ of $P_A$. Then we have $e\cdot[x,y]=[e\cdot x,e\cdot y]$ for all $x,y\in P$.
\end{pro}
\begin{proof}
Since $4e\cdot [x,y]=4e^2\cdot [x,y]$, by Eq.~\eqref{transposedpoissonlu}, we have
\begin{eqnarray*}
&&2[e\cdot x,y]+2[x,e\cdot y]=2e\cdot\big([e\cdot x,y]+[x,e\cdot y]\big)=[e\cdot x,y]+[x,e\cdot y]+2[e\cdot x,e\cdot y].
\end{eqnarray*}
Hence, we have
\begin{eqnarray*}
&&2[e\cdot x,e\cdot y]=[e\cdot x,y]+[x,e\cdot y]=2e\cdot [x,y].
\end{eqnarray*}
This completes the proof.
\end{proof}

\begin{rmk}\label{Pehomomorphism}
Proposition~\ref{idempotentlpro} shows that for a transposed Poisson algebra $(P,\cdot,[\cdot,\cdot])$, the linear map $P_e$ is not only a $\frac{1}{2}$-derivation~\cite{1-2}, but also a homomorphism of the Lie algebra $(P,[\cdot,\cdot])$.
\end{rmk}

Recall that a \textbf{Hom-Lie algebra}~\cite{HomLie} is a triple $(P,[\cdot,\cdot],\varphi)$ where $P$ is a vector space, $[\cdot,\cdot]:P \otimes P \to P$ is skew-symmetric and bilinear and $\varphi:P \to P$ is a linear map satisfying
\begin{equation}
[\varphi(x), [y,z]] + [\varphi(y), [z,x]] + [\varphi(z), [x,y]] = 0 \quad \text{for all}\: x,y,z \in P.
\end{equation}
Moreover, if $\varphi$ is an algebra homomorphism, then the Hom-Lie algebra is called \textbf{multiplicative}.
By \cite{Bai},  for any $x \in P$,  $(P,[\cdot,\cdot], P_x)$ is a Hom-Lie algebra structure.
Consequently, by Proposition~\ref{idempotentlpro}, we obtain the following conclusion.

\begin{cor}
Let $(P,\cdot,[\cdot,\cdot])$ be a transposed Poisson algebra with an idempotent $e$ of $P_A$. Then $(P,[\cdot,\cdot],P_e)$ is a multiplicative Hom-Lie algebra.
\end{cor}

\begin{cor}\label{idempotentcor}
Let $(P,\cdot,[\cdot,\cdot])$ be a transposed Poisson algebra with an idempotent $e$ of $P_A$. Then we have $P_eQ_e=Q_eP_e=Q_e$.
\end{cor}
\begin{proof}
By Proposition~\ref{idempotentlpro}, we have $P_eQ_e=Q_eP_e$. Since $[e,x]=[e\cdot e,x]=2e\cdot[e,x]-[e,e\cdot x]=e\cdot [e,x]$, we have $Q_e=P_eQ_e=Q_eP_e$.
\end{proof}

\section{The nilpotency of transposed Poisson algebras}
In this section, we consider nilpotency and nilpotent radicals of transposed Poisson algebras.\vspa

\subsection{Engel's theorem and some constructions of nilpotent transposed Poisson algebras}
In this subsection, we prove Engel's theorem for transposed Poisson algebras. The proof is similar to that in~\cite[Theorem 2.13]{On the solvable Poisson algebras}.\vspa
\begin{thm}\label{Engelthm}
Let $(P,\cdot,[\cdot,\cdot])$ be a transposed Poisson algebra. Then $P$ is nilpotent if and only if $P_A$ and $P_L$ are nilpotent.
\end{thm}
\begin{proof}
If $P$ is nilpotent, it is obvious that $P_A$ and $P_L$ are nilpotent.

Assume that $P_A$ and $P_L$ are nilpotent. Then there exist non-negative integers $r$ and $s$ such that $P_L^s=P_A^r=0$. By the definition of the lower central series, it suffices to show that there exists a positive integer $k$ such that any product $T_1T_2\cdots T_k = 0$, where each $T_i$ is either $P_{x_i}$ or  $Q_{x_i}$ for some $x_i\in P$.

Observe that Eq.~\eqref{transposedpoissonlu} is equivalent to the relation
\[
Q_xP_y = 2P_yQ_x - Q_{x\cdot y} \qquad \text{for all}\: x,y\in P.
\]
Thus, for any $T = T_1T_2\cdots T_k$ that contains a factor $Q_{x_i}P_{x_{i+1}}$, we may rewrite it as
\begin{eqnarray*}
T &=& T_1T_2\cdots T_{i-1}\,Q_{x_i}P_{x_{i+1}}\,T_{i+2}\cdots T_k \\
&=& 2\,T_1T_2\cdots T_{i-1}\,P_{x_{i+1}}Q_{x_i}\,T_{i+2}\cdots T_k
      - T_1T_2\cdots T_{i-1}\,Q_{x_i\cdot x_{i+1}}\,T_{i+2}\cdots T_k .
\end{eqnarray*}
This step does not change the number of $Q$ in the product\delete{, we call it {\bf inversion}}. Repeating it, we can give it a regular form
\begin{eqnarray*}
T = a\,P_{j_1}\cdots P_{j_t}Q_{j_{t+1}}\cdots Q_{j_k} + W,
\end{eqnarray*}
where $a\in {\bf k}$ and $W$ is a linear combination of products of length strictly less than $k$. Notice that every term in $T$ contains $k-t$ times of $Q$. Consequently, if $k-t \ge s$, then we have $T=0$ directly. If $k-t < s$, we choose $k \ge rs$. In such a long product there must appear a factor of the form $P_{x_i}P_{x_{i+1}}\cdots P_{x_{i+r-1}}$. Therefore we have $T=0$ as well. This completes the proof.
\end{proof}

\begin{cor}\label{realEngelthm}
Let $(P,\cdot,[\cdot,\cdot])$ be a finite-dimensional transposed Poisson algebra. Then $P$ is nilpotent if and only if $P_x$ and $Q_x$ are nilpotent for all $x\in P$.
\end{cor}

\begin{proof}
The ``only if " part follows directly from Theorem~\ref{Engelthm}. Assume that $P_x$ and $Q_x$ are nilpotent for all $x\in P$. By Engel's theorem for Lie algebras and the well-known analogous result for finite-dimensional commutative associative algebras, we conclude that both $P_A$ and $P_L$ are nilpotent. Hence, $P$ is nilpotent by Theorem~\ref{Engelthm}. This completes the proof.
\end{proof}

\begin{rmk}
Corollary~\ref{realEngelthm} is an analogue of Engel's theorem for Lie algebras. Corollary~\ref{realEngelthm} for a finite-dimensional transposed Poisson algebra can also be proved by induction on $\operatorname{dim}(P)$ using Proposition~\ref{idealpro2} and Lemma~\ref{minimalideal-Annlem}, similar to \cite{Frattini of dialgebras}.
\end{rmk}

Every finite-dimensional dialgebra possesses a unique maximal solvable ideal~\cite{Frattini of dialgebras}. The same conclusion holds for nilpotent ideals.
\begin{cor}
Let $(P,\cdot,[\cdot,\cdot])$ be a transposed Poisson algebra. If $I$ and $J$ are nilpotent ideals of $P$, then $I+J$ is also a nilpotent ideal. In particular, there exists a unique maximal nilpotent ideal in a finite-dimensional transposed Poisson algebra.
\end{cor}

\begin{proof}
It is straightforward by Theorem~\ref{Engelthm}.
\end{proof}
Next, we present some constructions of nilpotent and solvable transposed Poisson algebras.

Recall that transposed Poisson algebras are closed under taking tensor products~\cite{Bai}. Let $(P_1, \cdot_1, [\cdot, \cdot]_1)$ and $(P_2, \cdot_2, [\cdot, \cdot]_2)$ be two transposed Poisson algebras. Define two binary operations $\cdot$ and $[\cdot, \cdot]$ on $P_1 \otimes P_2$ by
\begin{eqnarray}\label{tensor1}
&&(x_1 \otimes x_2) \cdot (y_1 \otimes y_2) = x_1 \cdot_1 y_1 \otimes x_2 \cdot_2 y_2,
\end{eqnarray}
\begin{eqnarray}\label{tensor2}
&&[x_1 \otimes x_2, y_1 \otimes y_2] = [x_1, y_1]_1 \otimes x_2 \cdot_2 y_2 + x_1 \cdot_1 y_1 \otimes [x_2, y_2]_2,
\end{eqnarray}
for all $x_1, y_1 \in P_1$, $x_2, y_2 \in P_2$. Then $(P_1 \otimes P_2, \cdot, [\cdot, \cdot])$ is a transposed Poisson algebra.
\begin{pro}\label{construct1}
 Let $(P_1, \cdot_1, [\cdot, \cdot]_1)$ and $(P_2, \cdot_2, [\cdot, \cdot]_2)$ be two transposed Poisson algebras. If $P_1$ or $P_2$ is nilpotent (resp. solvable), then $P_1 \otimes P_2$ is nilpotent (resp. solvable).
\end{pro}

\begin{proof}
For brevity, the subscripts $1$ and $2$ in the binary operations $\cdot$ and $[\cdot,\cdot]$ will be suppressed. Without loss of generality, assume that $P_1$ is nilpotent. We claim that $(P_1 \otimes P_2)^{k}\subseteq P_1^{k} \otimes P_2$ for all $k\ge 0$. It is obvious when $k=0$. Suppose by induction that $(P_1 \otimes P_2)^{i}\subseteq P_1^{i} \otimes P_2$ holds for all $i\le k$. Now consider $i=k+1$. Using Eqs.~\eqref{tensor1} and~\eqref{tensor2}, we obtain
\begin{eqnarray*}
(P_1 \otimes P_2)^{k+1}&=&[(P_1 \otimes P_2)^{k},P_1 \otimes P_2]+(P_1 \otimes P_2)^{k}\cdot (P_1 \otimes P_2)\\
&\subseteq&(P_1\cdot P_1^k)\otimes P_2+[P_1,P_1^k]\otimes P_2+(P_1\cdot P_1^k)\otimes P_2\\
&\subseteq& P_1^{k+1} \otimes P_2.
\end{eqnarray*}
Thus, $P_1 \otimes P_2$ is nilpotent. The solvable case can be proved similarly.
\end{proof}

Recall \cite{Bai} that there is a natural construction of transposed Poisson algebras from commutative associative algebras with a derivation. Let $(P,\cdot)$ be a commutative associative algebra and $D$ be a derivation of $(P,\cdot)$.
Define the binary operation
\begin{eqnarray}\label{D-construct}
&&[x,y] := x \cdot D(y) - D(x) \cdot y \qquad \text{for all}~~x,~y \in P.
\end{eqnarray}
Then $(P,\cdot,[\cdot,\cdot])$ is a transposed Poisson algebra.

\begin{pro}\label{construct2}
Let $(P,\cdot)$ be a finite-dimensional commutative associative algebra and $D$ be a derivation of $(P,\cdot)$. Then the transposed Poisson algebra $(P,\cdot,[\cdot,\cdot])$ defined by Eq.~\eqref{D-construct} is nilpotent (resp. solvable) if and only if $P_A$ is nilpotent (resp. solvable).
\end{pro}
\begin{proof}
We only need to prove the necessity. First, we claim that $D(P_A^k)\subseteq P_A^k$ for all $k$. When $k=0$, the claim is obvious. Assume inductively that $D(P_A^i)\subseteq P_A^i$ holds for all $0\le i\le k$. For $i=k+1$, we have
\begin{eqnarray*}
D(P_A^{k+1})&=&D(P\cdot P_A^k)\\
&\subseteq & P\cdot P_A^k+D(P)\cdot P_A^k\\
&\subseteq &P_A^{k+1}.
\end{eqnarray*}
This completes the first claim.\par
Next, we claim that $P_L^k\subseteq P_A^k$ for all $k$. It is clear if $k=0$. Suppose that $P_L^i\subseteq P_A^i$ for all $0\le i\le k$. Then for $i=k+1$, using Eq.~\eqref{D-construct}, we obtain
\begin{eqnarray*}
P_L^{k+1}&=&[P,P_L^k]\subseteq [P,P_A^k]\\
&\subseteq& P\cdot D(P_A^k)+D(P)\cdot P_A^k\\
&\subseteq & P_A^{k+1}.
\end{eqnarray*}
This completes the second claim. Since $P_A$ is nilpotent, the inclusions $P_L^k\subseteq P_A^k$ for all $k$ imply that $P_L$ is also nilpotent. Consequently, $(P,\cdot,[\cdot,\cdot])$ is nilpotent by Theorem~\ref{Engelthm}.\par
The solvable case is clear since $P_A$ is nilpotent if and only if it is solvable.
\end{proof}

\subsection{Nilpotent radicals of transposed Poisson algebras}
For a finite-dimensional transposed Poisson algebra $(P,\cdot,[\cdot,\cdot])$, we denote by $Nil(P)$ the {\bf nilpotent radical} of $P$, that is, the maximal nilpotent ideal of $P$. Similarly, let $Nil_A(P)$ be the nilpotent radical of $P_A$, and $Nil_L(P)$ be the nilpotent radical of $P_L$. In this subsection, we will study their properties and relations.
\begin{lem}\label{minimalideal-Annlem}
Let $A$ be a minimal ideal and $N$ a nilpotent ideal of a transposed Poisson algebra $(P,\cdot,[\cdot,\cdot])$. If $A\subseteq N$, then $A\subseteq \operatorname{Ann}_P(N)$.
\end{lem}

\begin{proof}
By Proposition~\ref{idealpro2}, $[A,N]$ is an ideal of $P$. Since $[A,N]\subseteq A$, the minimality of $A$ implies that either $[A,N]=A$ or $[A,N]=0$. Because $N$ is Lie nilpotent, we must have $[A,N]=0$.

Moreover, since $(A\cdot N)\cdot P = A\cdot(N\cdot P) = A\cdot N$ and
\[
[A\cdot N,P] \subseteq A\cdot [N,P] + [N,A\cdot P] \subseteq A\cdot N + [N,A] \subseteq A\cdot N,
\]
it follows that $A\cdot N$ is also an ideal of $P$ contained in $A$. Thus, by the minimality of $A$, either $A\cdot N=0$ or $A\cdot N=A$. Since $N$ is associative nilpotent, we conclude that $A\cdot N=0$. Hence, $A\subseteq \operatorname{Ann}_P(N)$.
\end{proof}

\begin{thm}\label{Lienil-assonil}
Let $(P,\cdot,[\cdot,\cdot])$ be a finite-dimensional Lie nilpotent transposed Poisson algebra. Then $[P,P]$ is a nilpotent ideal.
\end{thm}

\begin{proof}
We prove it by induction on $\operatorname{dim}(P)$. If $\operatorname{dim}(P)=1$, then $[P,P]=0$ and the conclusion holds trivially. Assume that the conclusion holds for all Lie nilpotent transposed Poisson algebras of dimensions at most $k$. Suppose that $\operatorname{dim}(P)=k+1$.\par
If $[P,P]=0$, there is nothing to prove. Otherwise, by Proposition~\ref{idealpro2}, there exists a minimal ideal $B$ of $P$ contained in $[P,P]$. Consider the quotient $[P,P]/B$. By the induction hypothesis, $[P,P]/B$ is nilpotent. Hence there exists a non-negative integer $m$ such that $[P,P]_A^m\subseteq B$. By the proof of Lemma~\ref{minimalideal-Annlem}, we have $B\subseteq \operatorname{Ann}_L(P)$. Therefore, we obtain
\begin{eqnarray*}
[P,P]_A^{m+1}= [P,P]\cdot [P,P]_A^m\subseteq [P,P]\cdot B\subseteq [B\cdot P,P]\subseteq [B,P]=0.
\end{eqnarray*}
Thus $[P,P]$ is associative nilpotent. Hence it is nilpotent by Theorem~\ref{Engelthm}. This completes the induction and the proof.
\end{proof}

\begin{rmk}
For a transposed Poisson algebra $(P,\cdot,[\cdot,\cdot])$, Eq.~\eqref{transposedpoissonlu} and Proposition~\ref{idealpro3} lead to the inclusion $\operatorname{Ann}_L(P)\subseteq \operatorname{Ann}_A([P,P])$ directly. Thus, if $\operatorname{Ann}_L(P)\neq 0$, then $\operatorname{Ann}_A([P,P])\neq 0$. In fact, Theorem \ref{Lienil-assonil} gives a stronger statement: for a finite-dimensional Lie nilpotent transposed Poisson algebra $(P,\cdot,[\cdot,\cdot])$ ($\operatorname{Ann}_L(P)\neq 0$),
$[P,P]$ is a nilpotent ideal, which means that $\operatorname{Ann}_A([P,P])\neq 0$.

\end{rmk}

\begin{cor}\label{N=N_A}
Let $(P,\cdot,[\cdot,\cdot])$ be a finite-dimensional Lie nilpotent transposed Poisson algebra. Then $Nil_A(P)$ is a nilpotent ideal. Therefore, $Nil(P)=Nil_A(P)$.
\end{cor}

\begin{proof}
Obviously, we have $Nil(P)\subseteq Nil_A(P)$.
By Theorem~\ref{Lienil-assonil}, we obtain
\begin{eqnarray*}
[P,Nil_A(P)]\subseteq [P,P]\subseteq Nil(P)\subseteq Nil_A(P).
\end{eqnarray*}
Thus $Nil_A(P)$ is a nilpotent ideal of $P$. Hence $Nil_A(P)\subseteq Nil(P)$. This completes the proof.
\end{proof}

\begin{ex}
By \cite{3-dimensional transposed}, $(P=\mathbb{C}e_1\oplus\mathbb{C}e_2\oplus\mathbb{C}e_3,\cdot,[\cdot,\cdot])$ is a transposed Poisson algebra with the non-zero commutative associative products and the non-zero Lie bracket given by
\begin{eqnarray*}
 && e_1 \cdot e_2 = e_1,\:e_2\cdot e_2=e_2,\:e_2\cdot e_3=e_3,\:  [e_1, e_2] = e_3
\end{eqnarray*}
It is easy to see that $P_L$ is nilpotent and  $Nil(P)=Nil_A(P)=\mathbb{C}e_1\oplus\mathbb{C}e_3$.
\end{ex}

\section{Frattini theory of transposed Poisson algebras }
In this section, we introduce the Frattini theory of transposed Poisson algebras.
\vspace{-.2cm}

\subsection{Frattini subalgebras and Frattini ideals of transposed Poisson algebras}
We first recall the definitions of Frattini subalgebras and Frattini ideals of a dialgebra.
\begin{defi}\cite{Frattini of dialgebras}
For a dialgebra $P$, the {\bf Frattini subalgebra} \( F(P) \) is defined as the intersection of all maximal subalgebras of \( P \), and the {\bf Frattini ideal} \( \phi(P) \) is the largest ideal of \( P \) contained in \( F(P) \).
\end{defi}
\begin{rmk}
For a transposed Poisson algebra $(P,\cdot,[\cdot,\cdot])$, if $P_A$ is nilpotent, we can construct a one-dimensional subalgebra of $P$ by choosing an element $x \in \operatorname{Ann}_A(P)$. If $P_A$ is not nilpotent, then we can construct a one-dimensional subalgebra of $P$ by taking an idempotent element in $P_A$. Consequently, for any transposed Poisson algebra of dimension greater than one, a non-trivial maximal subalgebra must exist.
\end{rmk}

We first establish basic properties of Frattini subalgebras and Frattini ideals in dialgebras, which correspond to those in Lie algebras~\cite{Frattini of Lie algebras, Lie2, Lie3}. The proofs follow arguments from group theory.

\begin{pro}\label{Fpro1}
Let \( A\) be a subalgebra of the dialgebra $(P,\cdot,[\cdot,\cdot])$. Then we have the following conclusions.
\begin{enumerate}
\item\label{upper bound} $F(P),\phi(P)\subseteq P^1$.
\item $J(P)\subseteq P^{1}$, where $J(P)$ is the intersection of all maximal ideals of $P$.
    \item If \( A + F(P) = P \), then we have \( A=P \).
    \item If $A$ is an ideal of $P$, then $A\not\subseteq F(P) $ if and only if there exists a proper subalgebra $B$ of $P$ such that $P=A+B$.
\end{enumerate}
\end{pro}
\begin{proof}
They are straightforward. Here we only show the proof of \eqref{upper bound}. If $F(P)\not\subseteq P^1$, then there exists $x\in F(P)$ but $x\notin P^1$. We can construct a subspace $V$ of $P$ of codimension 1 that does not contain $x$. But $V$ is a maximal subalgebra of $(P,\cdot,[\cdot,\cdot])$ since $P^1\subseteq V$, which contradicts with $x\in F(P)$.
\end{proof}

Proposition~\ref{Fpro1} provides an upper bound for $F(P)$ and $\phi(P)$ with respect to the inclusion. Next, we give a lower bound for $F(P)$. \delete{Moreover, all conclusions can become equalities for some certain transposed Poisson algebras.}

\begin{pro}\label{F(P)lower bounds}
Let $(P,\cdot,[\cdot,\cdot])$ be a transposed Poisson algebra. Then we have $Ann_P(P)\cap P^1\subseteq F(P)$.
\end{pro}
\begin{proof}
When $Ann_P(P)\cap P^1=0$, the conclusion holds obviously.\par
Now, we suppose that $Ann_P(P)\cap P^1\neq 0$. For any non-zero element $x\in Ann_P(P)\cap P^1$, it can be written in the form $x=\sum_{i\in I}[r_i,s_i]+\sum_{j\in J}r'_j\cdot s'_j$ or $x=\sum_{i\in I}[r_i,s_i]$ or $x=\sum_{j\in J}r'_j\cdot s'_j$, where $r_i,s_i,r'_j,s'_j\in P\setminus Ann_P(P)$. Thus $r_i,s_i,r'_j,s'_j \notin {\bf k}x$ for each $i\in I,j\in J$.  Moreover, we can also assume that $\pi(r_i)=0$, $\pi(s_i)=0$, $\pi(r'_j)=0$ and $\pi(s'_j)=0$, where $\pi: P\rightarrow Ann_P(P)$ is the linear projection.  \par
Here we only consider the first form. If $x\notin F(P)$, we can choose a maximal subalgebra $M$ of $P$ which does not contain $x$. Therefore, not all $r_i,s_i,r'_j,s'_j$ are in $M$. Without loss of generality, we assume that $r_k\notin M$. Now we consider the subalgebra $\langle M,x\rangle$, which is generated by $M$ and $x$. Obviously, $P=\langle M,x\rangle=M\oplus{\bf k}x$ since $M$ is a maximal subalgebra and $x\in Ann_P(P)$, which contradicts the choice of $r_k$. Hence $x\in F(P)$. The other two forms are similar. Thus $Ann_P(P)\cap P^1\subseteq F(P)$. This completes the proof.
\end{proof}

\begin{ex}\label{ex0}
By \cite{3-dimensional transposed}, $(P=\mathbb{C}e_1\oplus\mathbb{C}e_2\oplus\mathbb{C}e_3,\cdot,[\cdot,\cdot])$ is a transposed Poisson algebra with the non-zero commutative associative product and the non-zero Lie bracket given by
\begin{eqnarray*}
 &&e_2 \cdot e_2 = e_3,\quad [e_1, e_2] = e_3.
\end{eqnarray*}
It is easy to see that \delete{$P$ is nilpotent and $F(P)=\phi(P)=Nil(P)^1=\mathbb{C}e_3$. Hence, we obtain that
\begin{eqnarray*}
&&Nil(P)^1 = Nil(P) \cap F(P) =P\cap \mathbb{C}e_3=\mathbb{C}e_3= \phi(P).
\end{eqnarray*}}
$Ann_P(P)\cap P^1=F(P)=\phi(P)=P^1=\mathbb{C}e_3$.
Thus, this example demonstrates that the upper bound and the lower bound for $F(P)$ may coincide.
\end{ex}

It was shown in \cite[Corollary 4.2]{Frattini of dialgebras} that the Frattini ideal of a Poisson algebra is  nilpotent. Next, we obtain a similar result for transposed Poisson algebras.

\begin{lem}\label{Frattiniideallem}
Let $(P,\cdot,[\cdot,\cdot])$ be a finite-dimensional transposed Poisson algebra. Suppose that $A$ is a subideal of $P$ and $B$ is an ideal of $A$ contained in $F(P)$. If the quotient $A/B$ is associative nilpotent, then $A$ is also associative nilpotent.
\end{lem}

\begin{proof}
Since $A$ is a subideal of $P$, there exists a chain of subalgebras of $P$ as follows:
\begin{eqnarray*}
&&0=A_0\subset A_1 \subset...\subset A_r=P,
\end{eqnarray*}
where $A_i$ is an ideal of $A_{i+1}$. Because $A/B$ is associative nilpotent, there exists a non-negative integer $s$ such that for every $a\in A$,
\begin{eqnarray*}
&&P_a^{r+s}(P)\subseteq P_a^{s}(A)\subseteq B\subseteq F(P).
\end{eqnarray*}
Consider the Fitting decomposition $P=P_a^{r+s}(P)+E_{P}^A(a)$. So we have $E_{P}^A(a)=P$ by Propositions~\ref{idealpro3} and \ref{Fpro1}. \delete{Otherwise, there would exist a maximal subalgebra $M$ containing $E_{P}^L(a)$, and then
\begin{eqnarray*}
&&P=P_a^{r+s}(P)+E_{P}^L(a)\subseteq F(P)+M\subseteq M,
\end{eqnarray*}
which is a contradiction. }Hence, $A$ is associative nilpotent.
\end{proof}

\begin{thm}\label{Frattiniidealthm}
Let $(P,\cdot,[\cdot,\cdot])$ be a finite-dimensional transposed Poisson algebra. Then $\phi(P)$ is associative nilpotent.
\end{thm}
\begin{proof}
It follows by taking \(A = B = \varphi(P)\) in Lemma~\ref{Frattiniideallem} directly.
\end{proof}

\subsection{Frattini subalgebras of nilpotent transposed Poisson algebras}

Note that by \cite{Abelian subalgebra of Poisson},  all maximal subalgebras are ideals in a nilpotent Poisson algebra. We obtain an analogous result for nilpotent transposed Poisson algebras by using the normalizer.
\begin{lem}\label{normalizerlem}
Let $(P,\cdot,[\cdot,\cdot])$ be a transposed Poisson algebra and let $A$ be a subalgebra of $P$. Then the normalizer $N(A)$ is a subalgebra of $P$.
\end{lem}

\begin{proof}
For all $x,y\in N(A)$, by Eq.~\eqref{transposedpoissonlu}, we have
\begin{eqnarray*}
[x,y]\cdot A+\big[[x,y],A\big]&\subseteq& [x\cdot A,y]+[x,y\cdot A]+\big[[x,A],y\big]+\big[[y,A],x\big]\\
&\subseteq&[A,y]+[x,A]\\
&\subseteq& A
\end{eqnarray*}
and
\begin{eqnarray*}
(x\cdot y)\cdot A+[x\cdot y,A]&\subseteq&x\cdot(y\cdot A)+x\cdot[y,A]+[y,x\cdot A]\\
&\subseteq&x\cdot A+[y,A]\\
&\subseteq&A.
\end{eqnarray*}
Hence, $x\cdot y,[x,y]\in N(A)$. This completes the proof.
\end{proof}

\begin{thm}\label{nil-all}
Let $(P,\cdot,[\cdot,\cdot])$ be a nilpotent transposed Poisson algebra. Then all maximal subalgebras of $P$ are ideals and $F(P)=\phi(P)=J(P)$.
\end{thm}

\begin{proof}
Since the normalizer of a proper subalgebra in a nilpotent dialgebra properly contains the subalgebra itself~\cite[Proposition 2.1]{Abelian subalgebra of Poisson}, it follows that for any maximal subalgebra $A$ of $P$, we have $N(A)=A$ by Lemma~\ref{normalizerlem}, that is, $A$ is an ideal of $P$. Thus, we obtain that $F(P)=\phi(P)=J(P)$.
\end{proof}
\vspa
Note that Example~\ref{ex0} is an example of Theorem~\ref{nil-all}. But the converse of Theorem~\ref{nil-all} is not true.
\begin{ex}\label{ex2}
By \cite{Bai}, $(P=\mathbb{C}e_1\oplus\mathbb{C}e_2,\cdot,[\cdot,\cdot])$ is a transposed Poisson algebra with the non-zero commutative associative product and the non-zero Lie bracket given by
\begin{eqnarray*}
&& e_1 \cdot e_1 = e_2,\quad [e_1, e_2] = e_2.
\end{eqnarray*}
It is easy to see that the only maximal subalgebra of $P$ is $\mathbb{C}e_2$, which is clearly an ideal. However, $P$ itself is not nilpotent.
\end{ex}

\vspa
\begin{thm}\label{allthm1}
Let $(P,\cdot,[\cdot,\cdot])$ be a finite-dimensional transposed Poisson algebra. If all maximal subalgebras of $P$ are ideals, then $P_A$ is nilpotent or $P={\bf k}e\oplus (1-e)\cdot P$, where $e$ is an idempotent of $P_A$ and $(1-e)\cdot P=\{x-x\cdot e | x\in P\}$.
\end{thm}

\begin{proof}
If $P_A$ is not nilpotent, then there exists an idempotent $e$ in $P_A$. We first claim that $(1-e)\cdot P$ and $A := \mathbf{k}e + (1-e)\cdot P\neq 0$ are subalgebras of $P$. By Proposition~\ref{idempotentlpro}, for all $x, y \in P$, we have
\begin{align*}
[x - e\cdot x, y - e\cdot y] &= [x,y] - [e\cdot x, y] - [x, e\cdot y] + [e\cdot x, e\cdot y] \\
&= [x,y] - 2e\cdot [x,y] + e\cdot [x,y] \\
&= [x,y] - e\cdot [x,y] \in (1-e)\cdot P.
\end{align*}
Moreover, by Corollary~\ref{idempotentcor}, we have $[{\bf k} e, y - e\cdot y] = {\bf k}[e, y] - {\bf k}[e, e\cdot y] = 0$. Also, $(x - e\cdot x)\cdot (y - e\cdot y) = x\cdot y - e\cdot x\cdot y \in (1-e)\cdot P$ and $(x - e\cdot x)\cdot e = 0$, for all $x,y\in P$. This proves the claim. Moreover, $\mathbf{k}e + (1-e)\cdot P$ holds as a direct sum not only of vector spaces, but also of algebras. \par
If $A \neq P$, then there exists a maximal subalgebra $M$ of $P$ containing $A$. By assumption, $M$ is an ideal of $P$. Consequently, for any $x \in P$, we have $x = e \cdot x + (x - e \cdot x) \in M$ for every $x \in P$, which implies $M = P$. This contradicts with that $M$ is a proper subalgebra. Therefore, we must have $A = \mathbf{k}e \oplus (1-e)\cdot P = P$. This completes the proof.
\end{proof}

\begin{rmk}\label{allthmremark1}
Under the assumption that every maximal subalgebra of $P$ is an ideal, $(1-e)\cdot P$ must be an ideal of $P$.
\end{rmk}
 Recall that an idempotent \( e \) is {\bf principal} if there is no idempotent \( u \) orthogonal to \( e \) (i.e., \( ue = eu = 0 \) with \( u^2 = u \neq 0 \)). If \( (P, \cdot) \) is not a nilpotent algebra, it has a principal idempotent element. In Theorem~\ref{allthm1}, if we take $e$ as a principal idempotent, we obtain the following corollary.
\begin{cor}\label{allcor1}
Let $(P,\cdot,[\cdot,\cdot])$ be a finite-dimensional Lie nilpotent transposed Poisson algebra. Then all maximal subalgebras of $P$ are ideals if and only if $P$ is nilpotent or $P={\bf k}e\oplus Nil(P)$, where $e$ is a principal idempotent.
\end{cor}

\begin{proof}
We first prove the necessity. If the associative algebra $P_A$ is nilpotent, then $P$ is nilpotent by Theorem~\ref{Engelthm}.
If $P_A$ is not nilpotent, we can choose a principal idempotent $e$ of $P_A$. By Theorem~\ref{allthm1}, we obtain that $P = {\bf k}e \oplus (1-e)\cdot P$, where $(1-e)\cdot P$ is an ideal of $P$ by Remark~\ref{allthmremark1}. We claim that $(1-e)\cdot P$ is associative nilpotent. If not, there exists an idempotent $e_1 \in (1-e)\cdot P$, i.e., $e_1 = (1-e)\cdot x$ for some $x \in P$. But then $e_1 \cdot e = (1-e)\cdot x \cdot e = 0$,
which contradicts with the fact that $e$ is a principal idempotent. Hence, $(1-e)\cdot P$ is associative nilpotent. By Theorem~\ref{Engelthm}, we obtain that $(1-e)\cdot P$ is nilpotent, so $(1-e)\cdot P=Nil(P)$. \par
Next, we prove the sufficiency. If $P_A$ is nilpotent, the conclusion follows by Theorem~\ref{nil-all}. Now suppose $P = {\bf k}e \oplus Nil(P)$. Then every maximal subalgebra of $P$ has the form $Nil(P)$ or ${\bf k}e \oplus M$, where $M$ is a maximal subalgebra of $Nil(P)$. Since $Nil(P)$ is nilpotent, $M$ is in fact an ideal of $Nil(P)$ by Theorem~\ref{nil-all}. Hence, they are ideals of $P$ by the condition $e \cdot Nil(P) = [e, Nil(P)] = 0$.
\end{proof}

Note that Example~\ref{ex2} is an example of Corollary~\ref{allcor1} where $P_A$ is nilpotent. Next, we give an example of Corollary~\ref{allcor1} where $P_A$ is not nilpotent.

\begin{ex}\label{ex5}
By \cite{Bai}, $(P=\mathbb{C}e_1\oplus\mathbb{C}e_2,\cdot,[\cdot,\cdot])$ is a transposed Poisson algebra with the trivial Lie bracket and the non-zero commutative associative product given by
\begin{eqnarray*}
 &&e_1 \cdot e_1 = e_1.
\end{eqnarray*}
It is easy to see that the maximal subalgebras of $P$ are $\mathbb{C}e_1$ and $\mathbb{C}e_2$, which are clearly ideals. Moreover, we have $P=\mathbb{C}e\oplus (1-e)\cdot P$, where the idempotent $e$ is $e_1$ and $(1-e)\cdot P=\mathbb{C}e_2=Nil(P)$.
\end{ex}
\vspb

By Proposition ~\ref{Fpro1}, we have $F(P) \subseteq P^1$. Under the stronger assumption that $P$ is nilpotent, the  reverse inclusion also holds. The proof is similar to that in ~\cite[Theorem 4.7]{Abelian subalgebra of Poisson}.\vspa
\begin{thm}\label{nilFP}
Let $(P,\cdot,[\cdot,\cdot])$ be a nilpotent transposed Poisson algebra. Then $F(P)=P^1$.
\end{thm}

\begin{proof}
For any maximal subalgebra \(M\) of \(P\), since \(P\) is nilpotent, there exists a non-negative integer $n$ such that \(P^n \not\subseteq M\) but \(P^{n+1}\subseteq M\). By Proposition~\ref{idealpro1} and the maximality of \(M\), we obtain \(P = M + P^n\). Moreover,  \(M\) is also an ideal of \(P\) by Theorem~\ref{nil-all}. Therefore, we have
\begin{eqnarray*}
&&P^1=P\cdot P + [P,P] = (M+P^n)\cdot (M+P^n) + [M+P^n,M+P^n]\subseteq M + P^{n+1} \subseteq M.
\end{eqnarray*}
Thus \(P^1 \subseteq F(P)\) and it follows that \(F(P) = P^1\) by Proposition~\ref{Fpro1} .
\end{proof}

Note that Example~\ref{ex0} provides an example of Theorem~\ref{nilFP}.

\delete{\begin{ex}\label{ex4}
Let $(P,\cdot,[\cdot,\cdot])$ be a $2$-dimensional vector space over $\mathbb{C}$ with basis $\{e_1, e_2, e_3\}$, where $P_L$is trivial and the non-zero commutative associative product is given by
\begin{eqnarray*}
 &&e_1 \cdot e_1 = e_2.
\end{eqnarray*}
According to the classification in \cite{Bai}, this defines a transposed Poisson algebra. A direct computation shows that $P$ is nilpotent and the maximal subalgebras of $P$ only $\mathbb{C}e_2$, so $F(P)=\mathbb{C}e_2=P^1$.
\end{ex}
}

\begin{lem}\cite[Lemma 3.1]{Frattini of dialgebras}\label{Frattiniideallem1}
If $C$ is a subalgebra of the dialgebra $A$, and $B$ is an ideal of $A$ contained in $F(C)$, then $B$ is contained in $F(A)$.
\end{lem}

\begin{cor}\label{F(P)lower bounds-Nil}
Let $(P,\cdot,[\cdot,\cdot])$ be a transposed Poisson algebra. Write $N = Nil(P)$ for brevity. Then we have $N^1\subseteq N\cap F(P)\subseteq \phi(P)$.
\end{cor}
\begin{proof}
By Proposition~\ref{idealpro1}, Lemma~\ref{Frattiniideallem1} (let $B=N^1,C=N,A=P$ ) and Theorem~\ref{nilFP}, we have $F(N)=N^1\subseteq F(P)$. Therefore, we obtain $N^1\subseteq F(P)\cap N$. We claim that $K\coloneqq F(P)\cap N$ is an ideal of $P$. Since $K\cdot P+[K,P]\subseteq N\cdot P+[N,P]\subseteq N$, we only need to prove that $K\cdot P+[K,P]\subseteq F(P)$. If not, there is a maximal subalgebra $M$ of $P$ such that $K\cdot P+[K,P]\not\subseteq M$, which also implies that $N\not\subseteq M$. So $P=N+M$ by the maximality of $M$. Since $K\subseteq N$ and $K\subseteq F(P)\subseteq M$, we obtain
\begin{eqnarray*}
&&K\cdot P+[K,P]=K\cdot(M+N)+[K,M+N]\subseteq N^1+M\subseteq K+M\subseteq F(P)+M\subseteq M,
\end{eqnarray*}
which leads to a contradiction. Hence, $K$ is an ideal of $P$ and $K\subseteq \phi(P)$.
\end{proof}
Note that Example~\ref{ex0} provides an example of Corollary~\ref{F(P)lower bounds-Nil}.

\begin{defi}
The {\bf Frattini series} of a transposed Poisson algebra \((P,\cdot,[\cdot,\cdot])\) is the sequence of subalgebras \(\{F_i\}\) defined by \(F_0 = P\) and \(F_i\) is the Frattini subalgebra of \(F_{i-1}\), for \(i > 1\). The Frattini index of \(P\) is the smallest integer \(f\) for which \(F_f = 0\).
\end{defi}
\vspb
For a finite-dimensional transposed Poisson algebra, since each $F_{i+1}$ is a proper subalgebra of $F_i$, the Frattini index exists and is at most the dimension of $P$.  For a nilpotent transposed Poisson algebra, we obtain a stronger bound.
\vspb
\begin{cor}
Let \((P,\cdot,[\cdot,\cdot])\) be a nilpotent transposed Poisson algebra. Then \(F_{i} \subseteq P^{(i)}\) holds for all \(i \in \mathbb{N}\). In particular, the Frattini index \(f\) of \(P\) satisfies \(f \le d\), where \(d\) is the solvable index, that is, the smallest integer $m$ such that $P^{(m)}=0$.
\end{cor}

\begin{proof}
We prove the statement by induction on \(i\). For \(i=0\), the inclusion is clear. Now assume that \(F_i \subseteq P^{(i)}\) holds for some \( i \ge 1\). By Theorem~\ref{nilFP}, we have
\begin{eqnarray*}
&&F_{i+1}=F_i^1 = F_i\cdot F_i+[F_i,F_i] \subseteq \bigl(P^{(i)}\bigr)^1 = P^{(i+1)}.
\end{eqnarray*}
This completes the inductive step. We obtain the conclusion directly.
\vspa
\end{proof}
Example~\ref{ex2} shows that the converse of Theorem~\ref{nilFP} is not true. However, we can prove that $F(P)=P^1$ implies associative nilpotency.\vspa

\begin{thm}\label{FP-Pnil}
Let $(P,\cdot,[\cdot,\cdot])$ be a finite-dimensional Lie nilpotent transposed Poisson algebra. If $F(P)=P^1$, then $P$ is nilpotent.
\end{thm}

\begin{proof}
Since $F(P)=P^1$, every maximal subalgebra $M$ of $P$ satisfies $P\cdot P+[P,P]\subseteq M$. Therefore, all maximal subalgebras of $P$ are ideals. If $P_A$ is not nilpotent, then we have $P={\bf k}e\oplus Nil(P)$ by Theorem~\ref{allthm1}. Therefore, we obtain
\begin{eqnarray*}
&&P^1=[{\bf k}e+ Nil(P),{\bf k}e+ Nil(P)]+\big({\bf k}e+Nil(P)\big)\cdot \big({\bf k}e+Nil(P)\big)=Nil(P)^1+{\bf k}e
\end{eqnarray*}
and the maximal subalgebras of $P$ are of the form $Nil(P)$ or ${\bf k}e\oplus M'$, where $M'$ is a maximal subalgebra of $Nil(P)$. Consequently, we obtain
\begin{eqnarray*}
&&F(P)\subseteq \bigl(F(Nil(P))+\mathbf{k}e\bigr)\cap Nil(P)=F\big(Nil(P)\big).
\end{eqnarray*}
Since $Nil(P)$ is nilpotent, Theorem~\ref{nilFP} implies $F\big(Nil(P)\big)=Nil(P)^1$. Thus we have
\begin{eqnarray*}
&&F(P)\subseteq Nil(P)^1\subsetneqq Nil(P)^1+{\bf k}e=P^1,
\end{eqnarray*}
which contradicts with $F(P)=P^1$. Hence, $P$ is nilpotent by Theorem~\ref{Engelthm}.
\end{proof}

\subsection{Socles and zero socles of transposed Poisson algebras}
The socle and zero socle of a Poisson algebra were introduced in~\cite{Frattini of dialgebras}. In this subsection, we consider the socle and zero socle of a transposed Poisson algebra.
\begin{defi}~\cite{Frattini of dialgebras}
For a transposed Poisson algebra \((P,\cdot,[\cdot,\cdot])\), the sum of its minimal ideals is called the {\bf socle}, denoted \(\operatorname{Soc}(P)\). The sum of its minimal abelian ideals is called the {\bf zero socle}, denoted \(\operatorname{Zsoc}(P)\).
\vspa
\end{defi}

With a similar proof in ~\cite[Theorem 4.4]{Frattini of dialgebras}, we have the following theorem.

\begin{thm}\label{freethm2}
Let $(P,\cdot,[\cdot,\cdot])$ be a finite-dimensional transposed Poisson algebra. Then $\operatorname{Zsoc}(P)\subseteq Nil(P) \subseteq \operatorname{Ann}_P\big (\text{Soc}(P)\big)$. If $\phi(P)=0$, then we have $\operatorname{Zsoc}(P)= Nil(P)= \operatorname{Ann}_P\big (\text{Soc}(P)\big)$.
\end{thm}

Note that Example~\ref{ex5} provides an example of Theorem~\ref{freethm2} with $F(P)=\phi(P)=0$ and $\operatorname{Zsoc}(P)= Nil(P)= \operatorname{Ann}_P\big (\text{Soc}(P)\big)=\mathbb{C}e_2$.

\begin{lem}
\cite[Lemma 3.6]{Frattini of dialgebras}\label{freelem3}
Let $B$ be an abelian ideal of a dialgebra $A$ such that $B \cap \phi(A) = 0$. Then there is a subalgebra $C$ of $A$ such that $A = B \dot{+} C$.
\end{lem}

In the next theorem and corollary, we consider the complement of $\operatorname{Zsoc}(P)$ in a Lie nilpotent transposed Poisson algebra with $\phi(P)=0$.

\begin{thm}\label{freethm3}
Let $(P,\cdot,[\cdot,\cdot])$ be a finite-dimensional Lie nilpotent transposed Poisson algebra with $\phi(P)=0$ and $M$ be a minimal abelian ideal of $P$. Then there exists a subalgebra $B$ such that $P= M\dot{+} B$ with $\phi(B)=0$.
\end{thm}

\begin{proof}
By Lemma~\ref{freelem3}, there exists a subalgebra $B$ such that $P = M \dot{+} B$. We only need to prove that $\phi(B)=0$. First consider the subalgebra $K := M + \phi(B)$ of $P$. By Theorem~\ref{Frattiniidealthm}, $\phi(B)$ is associative nilpotent. Moreover, since $M$ is abelian, we have
\begin{eqnarray*}
&&K_A^{k+1} \subseteq M\cdot \big(\phi(B)\big)_A^{k} + \big(\phi(B)\big)_A^{k+1}.
\end{eqnarray*}
Therefore, $K$ is also associative nilpotent. By Theorem~\ref{Engelthm}, we obtain that $K$ is nilpotent.\par
Next, we claim that $K_1 := M \cdot \phi(B) + [M, \phi(B)]\subseteq M$ is an abelian ideal. Indeed, we obtain
\begin{eqnarray*}
K_1 \cdot P &=& \big(M \cdot \phi(B) + [M, \phi(B)]\big) \cdot P \\
            &\subseteq& M \cdot \phi(B) + [M \cdot P, \phi(B)] + [M, \phi(B) \cdot P] \\
            &\subseteq& K_1 + [M, \phi(B) \cdot (B+M)]\\
            &\subseteq& K_1,
\end{eqnarray*}
and
\begin{eqnarray*}
[K_1, P] &=& [M \cdot \phi(B), P] + \big[[M, \phi(B)], P\big] \\
         &\subseteq& \phi(B) \cdot [M, P] + [M, \phi(B) \cdot P] + \big[[M, P], \phi(B)\big] + \big[M, [\phi(B), P]\big] \\
         &\subseteq& K_1+[M, \phi(B) \cdot (B+M)]+ \big[M, [\phi(B), B+M]\big]\\
         &\subseteq& K_1.
\end{eqnarray*}
Thus $K_1$ is an ideal of $P$. Hence, we obtain the claim. By the minimality of $M$, we must have $K_1 = 0$ or $K_1 = M$. If $K_1 = M$, we obtain $M = K_1 = M \cdot K + [M, K] \subseteq K^1$, since $M$ is abelian. Repeating this step, we obtain that $M \subseteq K^i$ for every positive integer $i$, which contradicts with the nilpotency of $K$. Therefore, we have $K_1 = 0$. Consequently, $\phi(B)$ is not only an ideal of $B$ but also of $P$.\par
By Lemma~\ref{Frattiniideallem1} (with $A = P$, $B = \phi(B)$, $C = B$), we obtain that $\phi(B) \subseteq \phi(P) = 0$. This completes the proof.
\end{proof}

\begin{cor}\label{freecor1}
 Let $(P,\cdot,[\cdot,\cdot])$ be a finite-dimensional Lie nilpotent transposed Poisson algebra with $\phi(P)=0$. Then there exists a subalgebra $A$ such that $P= \operatorname{Zsoc}(P)\dot{+} A$ with $\phi(A)=0$.
\end{cor}

\begin{proof}
If $\operatorname{Zsoc}(P)=0$, then we may take $A = P$. Then the conclusion holds trivially.\par
If $\operatorname{Zsoc}(P) \neq 0$, choose a minimal ideal $A_1$ of $P$ contained in $\operatorname{Zsoc}(P)$. Since $\operatorname{Zsoc}(P)$ is abelian, $A_1$ is a minimal abelian ideal of $P$. By Theorem~\ref{freethm3}, there exists a subalgebra $B_1$ of $P$ such that $P = A_1 \dot{+} B_1$ and $\phi(B_1)=0=\phi(P/A_1)$. If $A_1 = \operatorname{Zsoc}(P)$, the conclusion follows directly. Otherwise, we can choose an ideal $A_2$ of $P$ that is contained in $\operatorname{Zsoc}(P)$ and contains $A_1$ such that $A_2/A_1$ is a minimal ideal of $P/A_1$. Then $\phi\big((P/A_1)/(A_2/A_1)\big) = \phi(P/A_2) = 0$. Repeating this process, we obtain a sequence $A_1 \subsetneqq A_2 \subsetneqq \dots \subsetneqq \operatorname{Zsoc}(P)$ with $\phi(P/A_i) = 0$ at each step, until $A_i = \operatorname{Zsoc}(P)$ and $\phi(P/A_i)=0$. This completes the proof.
\end{proof}

\begin{cor}\label{freecor2}
 Let $(P,\cdot,[\cdot,\cdot])$ be a finite-dimensional Lie nilpotent transposed Poisson algebra with $\phi(P)=0$. Then there exists a subalgebra $U$ such that $P= Nil_A(P)\dot{+} U$ with $[U,U]=0$.
\end{cor}

\begin{proof}
By Theorems~\ref{Lienil-assonil}, \ref{freethm2} and Corollaries~\ref{N=N_A}, \ref{freecor1}, there exists a subalgebra $U$ such that $P= Nil_A(P)\dot{+} U$. Moreover, we have $[P,P]\subseteq Nil(P)=Nil_A(P)$. Therefore, we obtain $[U,U]\subseteq U\cap Nil_A(P)=U\cap \operatorname{Zsoc}(P)=0$. This completes the proof.
\end{proof}

\delete{\begin{rmk}
For a Lie nilpotent transposed Poisson algebra $(P,\cdot,[\cdot,\cdot])$, by Lemma~\ref{minimalideal-Annlem} and following the proof of Theorem 4.3 in~\cite{Frattini of dialgebras}, we have $\phi(P)=0$ if and only if $P = \operatorname{Zsoc}(P) \oplus C$ for some subalgebra $C$ of $P$.
\end{rmk}}

\par
\noindent {\bf Acknowledgments.}
This research is supported by Zhejiang
Provincial Natural Science Foundation of China (No. Z25A010006) and
Natural Science Foundation of China (No. 12171129).

\smallskip

\noindent
{\bf Declaration of interests. } The authors have no conflicts of interest to disclose.

\smallskip

\noindent
{\bf Data availability. } No new data were created or analyzed in this study.

\UseRawInputEncoding

\end{document}